\documentclass{article}
\usepackage{mathrsfs}
\usepackage{bbm}
\usepackage{amssymb,amsmath,amsthm,cases}
\usepackage[noadjust]{cite}
\numberwithin{equation}{section}
\begin{document}
\title {The octonionic Bergman kernel for the half space}
\author{Jinxun Wang\thanks{\scriptsize Department of
Applied Mathematics, Guangdong University of Foreign Studies, Guangzhou 510006,
China. E-mail: wjx@gdufs.edu.cn}\,, ~Xingmin
Li\thanks{\scriptsize School of Computer Sciences, South China
Normal University, Guangzhou 510631, China. E-mail: lxmin57@163.com}}
\date{}
\maketitle
\newtheorem{defi}{Definition}[section]
\newtheorem{theo}{Theorem}[section]
\newtheorem{prop}{Proposition}[section]
\newtheorem{lemm}{Lemma}[section]
\newtheorem{coro}{Corollary}[section]

\noindent\textbf{Abstract:} We obtain the octonionic Bergman kernel for half space in
the octonionic analysis setting by two different methods. As a consequence, we unify the 
kernel forms in both complex analysis and hyper-complex analysis.

\section{Introduction}
In a recent paper \cite{WL} we derived the octonionic Bergman kernel for the unit ball,
based on our newly defined inner product of the octonionic Bergman space. Note that in complex analysis
the unit disc and the upper half space are holomorphically equivalent through Cayley transform. In octonionic
space there is also a similar Cayley transform mapping the unit ball onto the half space (cf. \cite{WLL}). So
the problems for the octonionic Bergman space on the half space in $\mathbb{R}^8$ naturally arise.
But unfortunately the octonionic Cayley transform is neither left $\mathbb{O}$-analytic
not right $\mathbb{O}$-analytic by our definition, and the octonions are neither commutative nor associative,
which bring barriers to the study of the problems in half space through Cayley transform.
Thus we need to investigate the case for half space by other ways.

First we accordingly have the following definition.

\begin{defi}[the octonionic Bergman space on the half space]
Let $\mathbb{R}_+^8:=\{x\in\mathbb{O}: \mbox{Re}\,x>0\}$ be the half space in $\mathbb{R}^8$,
the octonionic Bergman space $\mathcal{B}^2(\mathbb{R}_+^8)$ is the class of
left octonionic analytic functions $f$ on $\mathbb{R}_+^8$ satisfying
$$\int_{\mathbb{R}_+^8}|f|^2dV<\infty,$$
where $dV$ is the volume element on $\mathbb{R}_+^8$.
\end{defi}

The inner product is defined as the usual way.

\begin{defi}[inner product on $\mathcal{B}^2(\mathbb{R}_+^8)$]
Let $f, g\in\mathcal{B}^2(\mathbb{R}_+^8)$, we define
$$(f,g)_{\mathbb{R}_+^8}:=\frac{1}{\omega_8}\int_{\mathbb{R}_+^8}
\overline{g}f dV,$$ where $\omega_8=\frac{\pi^4}{3}$ is the surface
area of the unit sphere in $\mathbb{R}^8$.
\end{defi}

\vskip 0.3cm
By density argument and limit argument we prove the following main theorem of this paper:
\begin{theo}\label{main}
The octonionic Bergman kernel of $\mathcal{B}^2(\mathbb{R}_+^8)$ exists.
Let
$$B(x,a)=-2\frac{\partial}{\partial x_0}E(x+\overline{a}),$$
where $E(x)=\frac{\overline{x}}{|x|^8}$ is the octonionic Cauchy kernel, then $B(\cdot,a)$ is the desired
octonionic Bergman kernel, i.e., $B(\cdot,a)\in\mathcal{B}^2(\mathbb{R}_+^8)$,
and for any $f\in\mathcal{B}^2(\mathbb{R}_+^8)$ and any $a\in \mathbb{R}_+^8$, there holds the
following reproducing formula
$$f(a)=(f,B(\cdot,a))_{\mathbb{R}_+^8}.$$
Moreover, the kernel is unique.
\end{theo}

The rest of the paper is organized as follows. In Section 2 to make the paper self-contained we briefly review
the octonion algebra and octonionic analysis. In Section 3 we give two proofs of Theorem \ref{main}. In the
last section we point out that the Bergman kernels can be unified in one form in both complex analysis and
hyper-complex analysis.

\section{The octonions and the octonionic analysis}
Octonions, which are also called Cayley numbers or the Cayley algebra, were discovered independently by
John T. Graves and Arthur Cayley in 1843 and 1845.
In 1898, Hurwitz had proved that the real numbers
$\mathbb{R},$ complex numbers $\mathbb{C},$ quaternions $\mathbb{H}$
and octonions $\mathbb{O}$ are the only normed division algebras
over $\mathbb{R}$ (\cite{Hur}), with the embedding relation
$\mathbb{R}\subseteq \mathbb{C}\subseteq \mathbb{H}\subseteq
\mathbb{O}$.

Octonions are an 8 dimensional algebra over
$\mathbb{R}$ with the basis $e_0,e_1,\ldots,e_7$ satisfying
$$e_0^2=e_0,~e_ie_0=e_0e_i=e_i,~e_i^2=-1,~\text{for}~i=1,2,\ldots,7.$$
So $e_0$ is the unit element and can be identified with $1$.
Denote
$$W=\{(1,2,3),(1,4,5),(1,7,6),(2,4,6),(2,5,7),(3,4,7),(3,6,5)\}.$$
For any triple $(\alpha,\beta,\gamma)\in W$, we set $$e_\alpha
e_\beta=e_\gamma=-e_\beta e_\alpha,\quad e_\beta
e_\gamma=e_\alpha=-e_\gamma e_\beta,\quad e_\gamma
e_\alpha=e_\beta=- e_\alpha e_\gamma.$$
Then by distributivity for any $x=\sum_0^7 x_ie_i$,
$y=\sum_0^7 y_je_j \in \mathbb{O}$, the multiplication $xy$ is defined to be
$$xy:=\sum_{i=0}^7\sum_{j=0}^7x_iy_je_ie_j.$$

For any $x=\sum_0^7 x_ie_i \in \mathbb{O}$, $\mbox{Re}\,x:=x_{0}$ is
called the scalar (or real) part of $x$ and
$\overrightarrow{x}:=x-\mbox{Re}\,x$ is called its vector part.
$\overline{x}:=\sum_0^7x_i\overline{e_i}=x_0-\overrightarrow{x}$
and $|x|:=(\sum_0^7x_i^2)^\frac{1}{2}$ are respectively the conjugate and
norm (or modulus) of $x$, they satisfy: $|xy|=|x||y|,$
$x\overline{x}=\overline{x}x=|x|^2,$
$\overline{xy}=\overline{y}\,\overline{x}$ $(x,y\in \mathbb{O}).$ So if
$x\neq0,$ $x^{-1}=\overline{x}/{|x|^2}$ gives the inverse of $x.$

Octonionic multiplication is neither commutative nor associative. But
the subalgebra generated by any two elements is associative, namely, The octonions
are alternative. $[x, y, z]:=(xy)z-x(yz)$ is called the associator of $x, y, z\in
\mathbb{O},$ it satisfies (\cite{B1, Jaco})
$$
[x,y,z]=[y,z,x]=-[y,x,z], \quad [x,x,y]=[\overline{x},x,y]=0.
$$

As a generalization of complex analysis and quaternionic
analysis to higher dimensions, the study of octonionic analysis
was originated by Dentoni and Sce in 1973 (\cite{DS}), 
and was systematically investigated by Li et al since 1995 (\cite{Li}).
Octonionic analysis is a function theory on
octonionic analytic (abbr. $\mathbb{O}$-analytic) functions. Suppose $\Omega$ is
an open subset of $\mathbb{R}^8$,
$f=\sum_0^7f_je_j\in C^1(\Omega,\mathbb{O})$ is an octonion-valued function, if
$$Df=\sum_{i=0}^{7}e_{i}\frac{\partial f}{\partial x_{i}}
=\sum_{i=0}^{7}\sum_{j=0}^{7}\frac{\partial f_j}{\partial x_{i}}e_ie_j=0$$
$$\left(fD=\sum_{i=0}^{7} \frac{\partial f}{\partial x_{i}}e_{i}=
\sum_{i=0}^{7}\sum_{j=0}^{7}\frac{\partial f_j}{\partial x_{i}}e_je_i=0\right),$$ then $f$
is said to be left (right) $\mathbb{O}$-analytic in $\Omega$, where
the generalized Cauchy--Riemann operator $D$ and its conjugate $\overline{D}$ are defined
by $$D:=\sum_{i=0}^7e_i\frac{\partial}{\partial x_i},~~\overline{D}:=
\sum_{i=0}^7\overline{e_i}\frac{\partial}{\partial x_i}$$ respectively.
A function $f$ is $\mathbb{O}$-analytic means that $f$ is meanwhile left
$\mathbb{O}$-analytic and right $\mathbb{O}$-analytic. From
$$\overline{D}(Df)=(\overline{D}D)f=\triangle
f=f(D\overline{D})=(fD)\overline{D},$$ we know that any left (right)
$\mathbb{O}$-analytic function is always harmonic. In the sequel,
unless otherwise specified, we just consider the left
$\mathbb{O}$-analytic case as the right $\mathbb{O}$-analytic case is essentially the same.
A Cauchy-type integral formula for this setting is:

\begin{lemm}[Cauchy's integral formula, see \cite{DS,LP2}]
Let $\mathcal{M}\subset\Omega$ be an 8-dimensional, compact differentiable and oriented
manifold with boundary. If $f$ is left $\mathbb{O}$-analytic in $\Omega$, then
\begin{align*}
f(x)=\frac{1}{\omega_8}\int_{y\in\partial\mathcal{M}}E(y-x)(d\sigma_yf(y)),\quad x\in
\mathcal{M}^o,
\end{align*}
where $E(x)=\frac{\overline{x}}{|x|^8}$ is the octonionic Cauchy kernel, $d\sigma_y
=n(y)dS$, $n(y)$ and $dS$ are respectively the outward-pointing unit normal vector
and surface area element on $\partial\mathcal{M}$,
$\mathcal{M}^o$ is the interior of $\mathcal{M}$.
\end{lemm}

Octonionic analytic functions have a close relationship with the
Stein--Weiss conjugate harmonic systems. If the
components of $F$ consist a Stein--Weiss conjugate
harmonic system on $\Omega\subset\mathbb{R}^8$,
then $\overline{F}$ is $\mathbb{O}$-analytic on $\Omega$.
But conversely this is not true (\cite{LP1}).
For more information and recent progress about octonionic analysis, we
refer the reader to \cite{Li, LPQ1, LPQ2, LW, LZP1, LZP2, LLW, WLL}.

\section{Two proofs of Theorem \ref{main}}
\subsection{Proof of Theorem \ref{main} by density}
For $\delta>0$ we let $\mathbb{R}_{\delta}^8=\{x\in\mathbb{O}: \mbox{Re}\,x>-\delta\}$.
The set of functions that are left octonionic analytic and square integrable on $\mathbb{R}_{\delta}^8$
is denoted by $\mathcal{B}^2(\mathbb{R}_{\delta}^8)$. Similar to the harmonic Bergman space (see \cite{ABR}), we have

\begin{lemm}\label{density}
$\bigcup_{0<\delta<1}\mathcal{B}^2(\mathbb{R}_{\delta}^8)$ is dense in $\mathcal{B}^2(\mathbb{R}_+^8)$.
\end{lemm}
\begin{proof}
Let $u\in\mathcal{B}^2(\mathbb{R}_+^8)$. For $\delta>0$, the function $x\mapsto u(x+\delta)$ belongs to
$\mathcal{B}^2(\mathbb{R}_{\delta}^8)$. But the functions $u(x+\delta)$ converge to $u(x)$ in $L^2(\mathbb{R}_+^8)$
as $\delta\rightarrow 0$.
\end{proof}

\begin{proof}[Proof of Theorem \ref{main}]
By lemma \ref{density} we only need to check the theorem for $\mathcal{B}^2(\mathbb{R}_{\delta}^8)$ for any $\delta>0$.
Now suppose $u\in\mathcal{B}^2(\mathbb{R}_{\delta}^8)$, we have
\begin{align}
&\frac{1}{\omega_8}\int_{\mathbb{R}_+^8}\overline{\frac{\partial}{\partial x_0}E(x+\overline{a})}u(x)dV_x\nonumber
\\=&\frac{1}{\omega_8}\int_{\mathbb{R}^7}\int_0^{+\infty}\overline{\frac{\partial}{\partial x_0}E(x_0+\overrightarrow{x}+\overline{a})}u(x_0+\overrightarrow{x})dx_0d\overrightarrow{x}.\label{int1}
\end{align}
After integrating by parts in the inner integral, (\ref{int1}) becomes
\begin{align}
&\frac{1}{\omega_8}\int_{\mathbb{R}^7}\left(\overline{E(x_0+\overrightarrow{x}+\overline{a})}u(x_0+\overrightarrow{x})\Big{|}_0^{+\infty}
-\int_0^{+\infty}\overline{E(x_0+\overrightarrow{x}+\overline{a})}\frac{\partial}{\partial x_0}u(x_0+\overrightarrow{x})dx_0\right)d\overrightarrow{x}\nonumber
\\=&\frac{1}{\omega_8}\left(-\int_{\mathbb{R}^7}\overline{E(\overrightarrow{x}+\overline{a})}u(\overrightarrow{x})d\overrightarrow{x}
-\int_{\mathbb{R}^7}\int_0^{+\infty}\overline{E(x_0+\overrightarrow{x}+\overline{a})}\frac{\partial}{\partial x_0}u(x_0+\overrightarrow{x})dx_0d\overrightarrow{x}\right)\nonumber
\\=&-u(a)-\frac{1}{\omega_8}\int_0^{+\infty}\left(\int_{\mathbb{R}^7}\overline{E(x_0+\overrightarrow{x}+\overline{a})}\frac{\partial}{\partial x_0}u(x_0+\overrightarrow{x})d\overrightarrow{x}\right)dx_0\label{int2},
\end{align}
where in the first equality we use the fact that $u(x_0+\overrightarrow{x})\rightarrow 0$ as $x_0\rightarrow +\infty$, in the second equality
we apply the Cauchy integral formula to $u$ on half space $\mathbb{R}_+^8$. After applying again the Cauchy integral formula to 
the function $w\mapsto \frac{\partial u}{\partial x_0}\big{|}_{x=x_0+w}$, (\ref{int2}) becomes
\begin{align*}
&-u(a)-\frac{1}{\omega_8}\int_0^{+\infty}\left(\int_{\mathbb{R}^7}\overline{E(\overrightarrow{x}+\overline{x_0+a})}\frac{\partial u}{\partial x_0}\Big{|}_{x=x_0+\overrightarrow{x}}d\overrightarrow{x}\right)dx_0
\\=&-u(a)-\int_0^{+\infty}\frac{\partial u}{\partial x_0}\Big{|}_{x=a+2x_0}dx_0
\\=&-u(a)-\frac{1}{2}u(a+2x_0)\big{|}_0^{+\infty}
\\=&-\frac{1}{2}u(a).
\end{align*}
So the octonionic Bergman kernel $B(x,a)$ exists and equals $-2\frac{\partial}{\partial x_0}E(x+\overline{a})$.
The proof of the uniqueness of the kernel is similar to the case for the unit ball (see \cite{WL}) which we omit here.
The proof of Theorem \ref{main} is therefore complete.
\end{proof}

\subsection{Proof of Theorem \ref{main} by taking limits}

Let $B_{p,r}$ be the ball centered at the point $p\in\mathbb{R}^8$ of radius $r$, $B_{p,r}(x,a)$ be the
octonionic Bergman kernel for $B_{p,r}$. The inner product on the octonionic Bergman space $\mathcal{B}^2(B_{p,r})$ is
defined by
$$(f,g)_{B_{p,r}}:=\frac{1}{\omega_8}\int_{B_{p,r}}\left(\overline{g(x)}\frac{\overline{x-p}}{|x-p|}\right)
\left(\frac{x-p}{|x-p|}f(x)\right)dV,$$
Then we have

\begin{lemm}\label{translation}
The octonionic Bergman kernel $B_{p,r}(x,a)$ is given by
$$B_{p,r}(x,a)=\frac{1}{r^8}B_{0,1}\left(\frac{x-p}{r},\frac{a-p}{r}\right),$$ where
$$B_{0,1}(x,a)=\frac{\left(6(1-|a|^2|x|^2)+2(1-\overline{x}a)\right)(1-\overline{x}a)}
{|1-\overline{x}a|^{10}}$$
is the octonionic Bergman kernel for the unit ball given in \cite{WL}.
\end{lemm}

\begin{proof}
If $f(x)\in\mathcal{B}^2(B_{p,r})$, then $f(rx+p)\in\mathcal{B}^2(B_{0,1})$. So for any $b\in B_{0,1}$ we have
$$f(rb+p)=\frac{1}{\omega_8}\int_{B_{0,1}}\left(\overline{B_{0,1}(x,b)}\frac{\overline{x}}{|x|}\right)
\left(\frac{x}{|x|}f(rx+p)\right)dV_x.$$
Let $rb+p=a$, $rx+p=y$, then $a, y\in B_{p,r}$, and the above reproducing formula becomes
$$f(a)=\frac{1}{r^8\omega_8}\int_{B_{p,r}}\left(\overline{B_{0,1}\left(\frac{y-p}{r},\frac{a-p}{r}\right)}\frac{\overline{y-p}}{|y-p|}\right)
\left(\frac{y-p}{|y-p|}f(y)\right)dV_y.$$
\end{proof}

Note that $\mathbb{R}_+^8=\bigcup_{r=1}^\infty B_{r,r}$. Similar to the harmonic Bergman space (see \cite{ABR}), we have
\begin{lemm}\label{limit}
The octonionic Bergman kernel $B(x,a)$ for $\mathbb{R}_+^8$ satisfies
$$B(x,a)=\lim_{r\rightarrow\infty}B_{r,r}(x,a)$$
for all $x, y\in\mathbb{R}_+^8$.
\end{lemm}

\begin{proof}[Proof of Theorem \ref{main}]
By Lemma \ref{translation} and Lemma \ref{limit} we get
\begin{align*}
B(x,a)=&\lim_{r\rightarrow\infty}\frac{1}{r^8}B_{0,1}\left(\frac{x-r}{r},\frac{a-r}{r}\right)
\\=&\lim_{r\rightarrow\infty}\frac{\left(6\left(1-\frac{|a-r|^2|x-r|^2}{r^4}\right)+2\left(1-\frac{\overline{x-r}(a-r)}{r^2}\right)\right)
\left(1-\frac{\overline{x-r}(a-r)}{r^2}\right)}{r^8\left|1-\frac{\overline{x-r}(a-r)}{r^2}\right|^{10}}
\\=&\lim_{r\rightarrow\infty}\frac{\left(6\left(r-\frac{|a-r|^2|x-r|^2}{r^3}\right)+2\left(r-\frac{\overline{x-r}(a-r)}{r}\right)\right)
\left(r-\frac{\overline{x-r}(a-r)}{r}\right)}{\left|r-\frac{\overline{x-r}(a-r)}{r}\right|^{10}}.
\end{align*}
But
\begin{align*}
&\lim_{r\rightarrow\infty}\left(r-\frac{\overline{x-r}(a-r)}{r}\right)
\\=&\lim_{r\rightarrow\infty}\frac{r^2-\overline{x-r}(a-r)}{r}
\\=&\lim_{r\rightarrow\infty}\frac{(a+\overline{x})r-\overline{x}a}{r}
\\=&a+\overline{x},
\end{align*}
and
\begin{align*}
&\lim_{r\rightarrow\infty}\left(r-\frac{|a-r|^2|x-r|^2}{r^3}\right)
\\=&\lim_{r\rightarrow\infty}\frac{r^4-|(a-r)(x-r)|^2}{r^3}
\\=&\lim_{r\rightarrow\infty}\frac{r^4-|r^2-(a+x)r+ax|^2}{r^3}
\\=&\lim_{r\rightarrow\infty}\frac{r^4-(r^2-(a+x)r+ax)(r^2-(\overline{a}+\overline{x})r+\overline{ax})}{r^3}
\\=&a+x+\overline{a}+\overline{x}
\\=&2\mbox{Re}(a+\overline{x}).
\end{align*}
Hence,
$$B(x,a)=\frac{\left(12\mbox{Re}(a+\overline{x})+2(a+\overline{x})\right)(a+\overline{x})}{|a+\overline{x}|^{10}}
=-2\frac{\partial}{\partial x_0}\left(\frac{\overline{x+\overline{a}}}{|x+\overline{a}|^8}\right).$$
\end{proof}

\section{Final remarks}
The arguments in the above proofs can be also applied to both complex analysis and
Clifford analysis, we therefore can unify the Bergman kernels and reproducing formulas
in complex and hyper-complex contexts. Let $\mathscr{A}$ denote the
complex algebra or hyper-complex algebra, i.e., $\mathscr{A}$ may
refer to complex numbers $\mathbb{C}$, quaternions $\mathbb{H}$,
octonions $\mathbb{O}$, or Clifford algebra $\mathscr{C}$. Assume
that the dimension of the underlying space of $\mathscr{A}$ is $m$,
$\mathbb{R}^m_+:=\{x\in\mathscr{A}: \mbox{Re}\,x>0\}$ is the half space
in $\mathbb{R}^m$. Then for any function $f$ which belongs to the Bergman space
$\mathcal{B}^2(\mathbb{R}^m_+)$ and any
point $a\in \mathbb{R}^m_+$, there holds
\begin{align*}
f(a)&=(f,B(\cdot,a))_{\mathbb{R}^m_+} \\
&=\frac{1}{\omega_m}\int_{\mathbb{R}^m_+}\overline{B(x,a)}f(x)dV
\\&=\frac{1}{\omega_m}\int_{\mathbb{R}^m_+}
\overline{-2\frac{\partial}{\partial x_0}\left(\frac{\overline{x+\overline{a}}}{|x+\overline{a}|^m}\right)}f(x)dV
\\&=\frac{1}{\omega_m}\int_{\mathbb{R}^m_+}
\frac{\left(2(m-2)\mbox{Re}(x+\overline{a})+2(x+\overline{a})\right)(x+\overline{a})}{|x+\overline{a}|^{m+2}}f(x)dV,
\end{align*}
where $\omega_m$ is the surface
area of the unit sphere in $\mathbb{R}^m$, $dV$ is the volume
element on $\mathbb{R}^m_+$.

\vskip 0.8cm \noindent{\Large\textbf{Acknowledgements}}

\vskip 0.3cm \noindent
This work was supported by the National Natural Science
Foundation of China (No. 11701105). The authors would like to thank
the anonymous referees for their helpful comments and suggestions.

\end{document}